\documentclass[a4paper,11pt]{article}

\PassOptionsToPackage{quiet}{fontspec}
\usepackage{bm}
\usepackage[english]{babel}
\usepackage{hyperref}
\usepackage{amssymb,amsmath,amsthm}
\usepackage{enumerate}
\usepackage{enumitem}
\usepackage{graphicx, color}
\usepackage{epsfig}
\usepackage{tikz}
\usetikzlibrary{calc,decorations.pathmorphing,decorations.text}
\usepackage{subcaption}
\usepackage{refcount}
\usepackage[hmargin=2.5cm,vmargin=3cm]{geometry}
\usepackage{mathtools, braket}
\usepackage[normalem]{ulem}
\usepackage{authblk}
\usepackage{centernot}
\usepackage{indentfirst}
\usepackage{framed} 
\allowdisplaybreaks[4]
\usepackage[nameinlink]{cleveref}
\usepackage{amsmath}
\usepackage{algorithm}
\usepackage{algorithmic}
\usepackage[numbers,sort&compress]{natbib}
\usepackage{hyperref}
\hypersetup{
	colorlinks=true,
	linkcolor=black,
	citecolor=black
}

\newtheorem{thm}{Theorem}[section]

\newtheorem{lem}[thm]{Lemma}
\newtheorem{corollary}[thm]{Corollary}

\renewcommand\ge\geqslant
\renewcommand\geq\geqslant
\renewcommand\le\leqslant
\renewcommand\leq\leqslant

\usepackage{caption}
\usetikzlibrary{matrix}
\usetikzlibrary{decorations.markings}
\tikzstyle{vertex}=[circle, draw, fill=black!50,
inner sep=0pt, minimum width=4pt]
\tikzset{->-/.style={decoration={
			markings,
			mark=at position .5 with {\arrow{>}}},postaction={decorate}}}

\tikzstyle{bigblue}=[color=blue, very thick, >=stealth]
\tikzstyle{lightblue}=[color=blue, thin, >=stealth]
\tikzstyle{bigred}=[color=red, very thick, >=stealth]
\tikzstyle{lightred}=[color=red, thin, >=stealth]
\tikzstyle{biggreen}=[color=black!30!green, very thick, >=stealth]
\tikzstyle{lightgreen}=[color=black!30!green,  thin, >=stealth]

\newlength{\bibitemsep}
\newlength{\bibparskip}
\let\oldthebibliography\thebibliography
\renewcommand\thebibliography[1]{
	\oldthebibliography{#1}
	\setlength{\parskip}{\bibitemsep}
	\setlength{\itemsep}{\bibparskip}
}

\title{Reconfiguration graphs of $K_{2,3}$-minor-free graphs}
\author{Ruijuan Gu\thanks{Sino-European Institute of Aviation Engineering, Civil Aviation University of China, Tianjin 300300, P.R.China. Email: millet90@163.com.},~ 
Hui Lei\thanks{School of Statistics and Data Science, LPMC and KLMDASR, Nankai University, Tianjin 300071, P.R.
China. Email: hlei@nankai.edu.cn.},~ Zhaoxiang Li\thanks{School of Mathematical Sciences and LPMC, Nankai University, Tianjin 300071, P.R.China. Email: zhaoxiangli@mail.nankai.edu.cn.},~ Yulai Ma\thanks{Center for Combinatorics and LPMC, Nankai University, Tianjin 300071, P.R.China. Email: yulai.ma@upb.de.},~ Susu Wang\thanks{Center for Combinatorics and LPMC, Nankai University, Tianjin 300071, P.R.China. Email: susuwang@mail.nankai.edu.cn.}
 
}

\date{\today}

\begin{document}
\maketitle
\begin{abstract}

The $\ell$-reconfiguration graph of a graph $G$, denoted by $\mathcal{R}_{\ell}(G)$, is the graph whose vertices are the proper $\ell$-colorings of $G$, with an edge between two colorings if they differ in color on exactly one vertex. For any graph $G$ of treewidth at most $2$, Bousquet and Perarnau showed that $\mathcal{R}_\ell(G)$ has linear diameter for $\ell\geq 6$. This result was later extended by Bartier, Bousquet, and Heinrich, who proved that $\mathcal{R}_5(G)$ also has linear diameter. 

In this paper, we show that for each $\ell\geq 5$,  the $\ell$-reconfiguration graphs of $K_{2,3}$-minor-free graphs,  some of which include graphs of treewidth $3$, have linear diameter. As a key step in our proof, we also establish that the $(\ell-1)$-reconfiguration graphs of cactus graphs have linear diameter.

\end{abstract}

\section{Introduction}
In the study of reconfiguration problems, two main types of questions arise. The first is whether, for a given graph class $\mathcal{G}$, there exists an integer $t$ such that $\mathcal{R}_{\ell}(G) $ is connected for each $G\in \mathcal{G}$ and ${\ell}\geq t$. The second concerns the diameter of $\mathcal{R}_{\ell}(G)$ when it is connected.

A graph $ G$ is {\it $d$-degenerate} if every induced subgraph contains a vertex of degree at most $d$. In $2007$, Cereceda  \cite{Cerecedacon2007} conjectured that for any $d$-degenerate graph $G$ and ${\ell}\geq d + 2$,
the diameter of $\mathcal{R}_{\ell}(G) $ is $O(n^2)$, and confirmed this for ${\ell}\geq 2d+1$.
Dyer et al. \cite{Dyerdeg2006} and Cereceda et al. \cite{Cerecedamix2009} independently proved that $\mathcal{R}_{\ell}(G)$ is connected for any $\ell\geq d+2$. In $2022$, Bousquet and Heinrich \cite{Bousquettrianglefree2022} established that for each $d$-degenerate graph $G$ and $\ell\geq d+2$, the diameter of $\mathcal{R}_{\ell}(G) $ is $O(n^{d+1})$. Regarding the linear bounds for the diameter,  Bousquet and Perarnau \cite{Bousquetsparse2016} showed that for each $d$-degenerate graph $G$ and $\ell\geq  2d+2$, the diameter of $\mathcal{R}_{\ell}(G) $ is at most $(d+1)\cdot n$.

It is well known that planar graphs are 5-degenerate, making Cereceda's conjecture particularly relevant for this class of graphs. Feghali
\cite{Feghalimad2021} proved that for any planar graph $G$ and  $\ell\geq 8$, the diameter of $\mathcal{R}_{\ell}(G) $ is $O(n^2)$. As for the case $\ell=7$, a result of Bousquet and Heinrich  \cite{Bousquettrianglefree2022} showed that the diameter of $\mathcal{R}_7(G) $ is $O(n^6)$. Concerning the linear bounds on the diameter, Dvo\v{r}\'{a}k and Feghali \cite{DvorakThomassen2021} proved that for each planar graph $G$, the diameter of $\mathcal{R}_{10}(G) $ is at most $8n$.

In addition, a special subclass of planar graphs is studied, namely the $K_4$-minor-free graphs. Note that a graph is $K_4$-minor-free if and only if its treewidth is at most $2$.
It was established in \cite{Bousquetsparse2016} that for any $K_4$-minor-free graph $G$ and any $\ell\geq6$, $\mathcal{R}_\ell(G)$ has linear diameter. 
Furthermore, Bartier et al. \cite{Bartiertw22021} showed that the same holds for $\ell=5$, i.e., $\mathcal{R}_5(G)$  has linear diameter as well. We restate these two results in detail below. An $\ell$-coloring $\alpha$ is said to be {\it transformed} to another $\ell$-coloring $\beta$ if there is a path from $\alpha$ to $\beta$ in $\mathcal{R}_\ell(G)$.

\begin{thm}\label{tw2}
Let $G$ be a $K_4$-minor-free graph and $\ell$ be an integer with $\ell\geq 5$. For every pair of $\ell$-colorings $\alpha$ and $\beta$, there exists a transformation from $\alpha$ to $\beta$ by recoloring each vertex at most $t$ times, where $t$ is as follows:
\renewcommand{\labelenumi}{(\arabic{enumi})}
\begin{enumerate}
    \item\label{k23-6} $t=3$ if $\ell\geq 6$; {\em \cite{Bousquetsparse2016}}
    \item\label{k23-5} $t=1386$ if $\ell=5$. {\em\cite{Bartiertw22021}}
\end{enumerate}
\end{thm}

It is known that a simple graph is outerplanar if and only if it is both $K_4$-minor-free and $K_{2,3}$-minor-free \cite{Chartrandouterplanar1967}. Motivated by this structural characterization, we investigate the diameter of the $\ell$-reconfiguration graphs of $K_{2,3}$-minor-free graphs for each $\ell\geq5$. Note that $K_{2,3}$-minor-free graphs have treewidth at most $3$, and some have treewidth exactly $3$.

\begin{thm}\label{K23}
Let $G$ be a $K_{2,3}$-minor-free graph and $\ell$ be an integer with $\ell\geq 5$. For every pair of $\ell$-colorings $\alpha$ and $\beta$, there exists a transformation from $\alpha$ to $\beta$ by recoloring each vertex at most $t$ times, where $t$ is as follows:
\renewcommand{\labelenumi}{(\arabic{enumi})}
\renewcommand{\labelenumi}{(\arabic{enumi})}
\begin{enumerate}
    \item\label{k23-6} $t=3$ if $\ell\geq 6$;
    \item\label{k23-5} $t=1387$ if $\ell=5$.
\end{enumerate}
\end{thm}

As a key step towards proving Theorem~\ref{K23}, we obtain the following result.

\begin{thm}\label{cactus}
Let $G$ be a cactus graph and $\ell$ be an integer with $\ell\geq 4$. For every pair of $\ell$-colorings $\alpha$ and $\beta$, there is a transformation from $\alpha$ to $\beta$ by recoloring each vertex at most $3$ times.  
\end{thm}

%\lzx{algorithm, Lemma~\ref{lem: 5 to 4} add edges}

The organization of this paper is as follows.  The next section presents some notation and terminology, as well as structural properties of $K_{2,3}$-minor-free graphs. Sections \ref{Cactus} and \ref{last} are devoted to the proofs of Theorems \ref{cactus} and \ref{K23}, respectively.
%An open problem is proposed in a subsequent section.

\section{Preliminaries}\label{preli}
All graphs in this paper are finite and simple. For undefined notation and terminology, we refer readers to \cite {Bondygraph2008}.  
 Let $G=(V(G), E(G))$ be a graph and $X\subseteq V(G)$.
%two vertices $u$ and $v$ are {\it adjacent} if $uv\in E(G)$. 
%For any vertex set $X$ of $G$, denote the number of vertices of $X$ by $|X|$. 
 The subgraph of $G$ {\it induced} by $X$ is denoted by $G[X]$. Particularly, the set $X$ is called {\it independent} if $G[X]$ contains no edges. For convenience, we simply write $G-X$ for $G[V(G)\setminus X]$. If $X=\{x\}$, then we write $G-x$ for $G-\{x\}$.  For any $x\in V(G)$, let $d_G(x)$ denote the degree of $x$. Let  $N_G(x)$ be the neighborhood of $x$ in $G$ and $N_{G}[x]:=N_{G}(x)\cup\{x\}$ be the closed neighborhood of $x$.  Sometimes we omit the subscript $G$ if there is no conflict occurs, such as using $d(x)$, $N(x)$ and  $N[x]$ instead of $d_G(x)$,  $N_G(x)$ and $N_{G}[x]$.
%we simply write $d(x)$, $N(x)$ and $N[x]$ instead of $d_G(x)$, $N_G(x)$ and $N_{G}[x]$.  
%A graph $G$ has a {\it minor} $H$ if $H$ can be obtained from a subgraph of $G$ by contracting edges, deleting vertices or deleting edges. A graph $G$ is called {\it $H$-minor-free} if $G$ does not contain $H$ as a minor. 
A {\it clique} of $G$ is an induced subgraph of $G$ that is complete, and a clique on $k$ vertices is denoted by $K_k$. The {\it clique number} of $G$, denoted by $\omega(G)$, is the maximum order of cliques in $G$. A vertex $x$ is called a {\it cut vertex} if $G-x$ has more components than $G$. A {\it block} is a connected graph with no cut vertices. A {\it block of the graph $G$} is a maximal subgraph with this property, and an {\it end block} is one that contains exactly one  cut 
vertex of $G$. A planar graph is called {\it outerplanar} if it can be embedded in the plane such that all its 
vertices lie on the boundary of a single face, which we hereafter assume to be the exterior face.  A {\it cactus graph} is a connected planar graph in which every block is either a cycle or an edge. The ordering $v_1,v_2,\ldots,v_n$ of the vertices of $G$ is called a {\it perfect elimination ordering} of $G$ if for each $i$, $v_i$ and its neighbors in $\{v_{i+1},v_{i+2},\ldots,v_n\}$ form a clique. A chordal graph is a graph without induced cycle of length at least four. Equivalently, a graph is chordal if and only if it has a perfect elimination ordering.
% We denote by $C_n:=x_1x_2\ldots x_nx_1$ the cycle on $n$ vertices, and by  $P_n:=x_1 x_2 \ldots x_n$ the path on $n$ vertices. A cycle of length $\ell$ is called an {\it $\ell$-cycle}, and a cycle of length at least $\ell$ is referred to as an {\it $\ell^+$-cycle}. 
% \yl{A cycle of length $k$ ({\it at least $k$ }, respectively) is called a {\it $k$-cycle} ({\it $k^+$-cycle}, respectively) and denoted by $C_{k}$ ({\it $C_{k^+}$}, respectively). Similarly, a path on $k$ vertices is denoted by $P_{k}$.}

A {\it proper 
$\ell$-coloring} of $G$ is a mapping $c: V(G)\rightarrow \{1,2,\ldots,\ell\}$ such that any two adjacent vertices of $G$ receive different colors.
%for each $i\in \{1,2,\ldots,k\}$, $c^{-1}(i)$ is an independent set, 
 All colorings considered in this paper are proper, so we use coloring instead of proper coloring in the remaining context. The {\it $\ell$-reconfiguration graph}, denoted by $\mathcal{R}_{\ell}(G)$, is the graph whose vertices are the $\ell$-colorings of $G$ and two colorings are joined by an edge if they differ in color on exactly one vertex. Let $\alpha$ be a coloring of $G$ and $G_1$ be a subgraph of $G$, we use $\alpha(v)$ to denote the color of a vertex $v$ and $\alpha|_{G_1}$the restriction of coloring $\alpha$ on $G_1$, respectively.
% When $G_1$ is a vertex $u$, we simply write $\alpha|_{u}$ to denote the coloring $\alpha$ on $u$.
%delete.If a graph $G$ can be recolored from an $\ell$-coloring $\alpha$ to a $k$-coloring $\beta$ by recoloring each vertex $h$ times, where $k\leq \ell$, then for each vertex $u$, we denote colors of $u$ in different recoloring steps as follows: let $c_j(u)$ be the color of $u$ at step $j$ for each $1\leq j\leq h$, where $c_h(u)=\beta(u)$. For convenience, let $c_0(u)=\alpha(u)$ be the color of $u$ at step $0$. Note that $c_j(u)$ may be equal to $c_{j+1}(u)$ for $0\leq j\leq h-1$.

For a digraph $D$, let $uv$  denote an arc with head $v$ and tail $u$. Precisely, an arc $uv$ is an {\it out-arc} of $u$ and an {\it in-arc}  of $v$. For a vertex $v$, denote the  {\it in-neighborhood} and {\it out-neighborhood} of $v$ by $N^-(v)$ and $N^+(v)$, where $N^-(v)=\{u\in V(D)|uv \text{ is an arc}\}$ and $ N^+(v)=\{u\in V(D)|vu \text{ is an arc}\}$. Moreover, the {\it closed in-neighborhood} and {\it closed out-neighborhood} of $v$ are denoted by $N^-[v]$ and $N^+[v]$, where $N^-[v]=\{v\}\cup N^-(v)$ and $N^+[v]=\{v\}\cup N^+(v)$. 
Let $d^+(v)=|N^+(v)|$ and $d^-(v)=|N^-(v)|$. A {\it dicycle} $C$ is an alternating sequence $C=u_0u_1\ldots u_ku_0$ such that $u_iu_{i+1}$ is an arc for every $i\in \{0,1,\ldots,k-1\}$, $u_ku_0$ is an arc and  all vertices are different. 
%If $u_0=u_k$, then we call it a {\it dicycle}, and denote by $C_k:= u_0u_1\ldots u_ku_0$.} \lzx{no dipath used}
Especially, call $uv$ a {\it digon} if both $uv$ and $vu$ are arcs. We consider a digon as a $2$-dicycle.
A {\it complete dicycle} is a dicycle with a digon between every two adjacent vertices in the dicycle.

In the next, we shall introduce some  structural characterizations of $K_{2,3}$-minor-free graphs, which are useful in the subsequent sections.

%It is observed that every outerplanar graph $G$ contains an independent set $I$ such that $G-I$ is a forest. We give a stronger result as follows. 

\begin{thm}[\cite{Yuk232012}]\label{thm: k23 of Yuk}
A graph $G$ is $K_{2,3}$-minor-free if and only if each block of $G$ is either a $K_4$ or an outerplanar graph. 
\end{thm} 

\begin{lem}\label{ind}
For any outerplanar graph $G$ and a vertex $u\in V(G)$, there is an independent set $I$ such that $u\in I$ and $G-I$ is a forest.
\end{lem}
\begin{proof}
We proceed by induction on $|V(G)|$.  It is trivial when $|V(G)|\leq 3$. Assume that Lemma~\ref{ind} holds for all outerplanar graphs with at most 
$k$ vertices. Now  consider the case $|V(G)|= k+1$, and fix a vertex  $u\in V(G)$. Since $G$ is an outerplanar graph, it contains at least two vertices of degree at most $2$. Let $v$ be one such vertex, distinct from 
$u$. Since $|V(G-v)|=k$, by the induction hypothesis, there is an independent set $I'\subseteq V(G-v)$ such that $u\in I'$ and $(G-v)-I'$ is a forest. If $N(v)\cap I'= \emptyset$, then let $I=I'\cup \{v\}$ and otherwise let $I=I'$. It can be verified that the independent set $I$ satisfies all the required properties, thus completing the proof.
\end{proof}

\begin{lem}\label{cactus-I}
Let $G$ be a $K_{2,3}$-minor-free graph. There exists an independent set $I$ such that 
$(1)$ $B_1-I$ is a forest for any block $B_1$ of $G$ except those isomorphic to $K_4$, and 
$(2)$ $|V(B_2)\cap I|=1$ for any block $B_2\cong K_4$ of $G$. 

%In other words, $G-I$ is a union of cactus graphs without $(\geq 4)$-cycle.
\end{lem}
\begin{proof}
Suppose, for contradiction,  that $G$ is a minimum counterexample with respect to the number of vertices. By the choice of $G$, $G$ is connected. By Theorem~\ref{thm: k23 of Yuk} and Lemma~\ref{ind}, %every $2$-connected $K_{2,3}$-minor-free graph is outerplanar or $K_4$, hence  
$G$ is neither an outerplanar graph nor isomorphic to $K_4$. This implies that $G$ has at least two blocks.
Let $B$ be an end block of $G$ and $v\in V(B)$ be a cut vertex of $G$. Define $G_1:=G-(V(B)\setminus\{v\})$. The minimality of $G$ implies that $G_1$ admits an independent set $I_1$ satisfying conditions (1) and (2).

By Theorem~\ref{thm: k23 of Yuk}, $B$ is isomorphic to $K_4$ or an outerplanar graph. We first consider the case  $B\cong K_4$. If $v\in I_1$, then let $I=I_1$; if $v\notin I_1$, then let $I=I_1\cup \{u\}$, where $u\in V(B-v)$. In either case, the independent set $I$ of $G$ satisfies  conditions (1) and (2), a contradiction.
Now consider the case that $B$ is outerplanar. Let $u\in V(B)$ be a vertex such that $u=v$ if $v\in I_1$ and $u\in N_B(v)$ if $v\notin I_1$. According to Lemma~\ref{ind}, there exists an independent set $I_2\subseteq B$ such that $u\in I_2$ and $B-I_2$ is a forest. It follows that the independent set $I_1\cup I_2$ of $G$ satisfies conditions (1) and (2), a contradiction. 
% such that for any vertex $u$, $u\in I_2$ and $B-I_2$ is a forest, and satisfies the following: if $v\in I_1$, then let $v\in I_2$; if $v\notin I_1$, then let $u\in I_2$ where $u\in N_B(v)$.
% Let $I=I_1\cup I_2$, and the independent set $I$ satisfying conditions (1) and (2), a contradiction. 
%If $v\notin I_1$, let $u\in I_2$ where $u\in N_B(v)$, since then $v\notin I_2$, let $I=I_1\cup I_2$, and the independent $I$ is what we want.
\end{proof}

% For a $K_{2,3}$-minor-free graph $G$, let $\alpha$  and $\beta$ be two $\ell$-colorings of $G$ where $\ell\geq 5$, if $\alpha$ can be transformed into an $(\ell-1)$-coloring $\alpha'$ by recoloring each vertex at most $h$ times, with a similar argument, we can transform $\beta$ into an $(\ell-1)$-coloring $\beta'$ by recoloring each vertex at most $h$ times. By Lemma~\ref{cactus-I}, the graph $G-I$ is a union of cactus graphs where $I$ is an independent set. Then we recolor $I$ to the color $\ell$ under $\alpha'$ and $\beta'$, by Theorem~\ref{cactus}, there is a transformation of $G-I$ from any two $(\ell-1)$-colorings by recoloring each vertex at most $3$ times, then we can get that there is a transformation of $G$ from $\alpha$ to $\beta$ by recoloring each vertex at most $2h+3$ times. Therefore, next, we prove that $\alpha$ can be transformed into an $(\ell-1)$-coloring $\alpha'$ by recoloring each vertex at most $h$ times.

\section{Proof of Theorem~\ref{cactus}}\label{Cactus} 
In this section, we  prove Theorem~\ref{cactus}, beginning with a key lemma that will be used throughout its proof. For any two colorings $\gamma_1$ and $\gamma_2$ of $G$, we construct a digraph $D_{\gamma_1\gamma_2}$ from $G$ such that $V(D_{\gamma_1\gamma_2})=V(G)$, and $xy$ is an arc in $D_{\gamma_1\gamma_2}$ if and only if $\gamma_1(y)=\gamma_2(x)$ and $x$ is adjacent to $y$ in $G$. 

\begin{lem}\label{Lemma-j=i+1} Let $D_{\gamma_1\gamma_2}$ be the digraph constructed from a graph $G$ and two colorings $\gamma_1$ and $\gamma_2$.
Then there exists a transformation from $\gamma_1$ to $\gamma_2$ in which each vertex $v\in V(G)$ is recolored at most once if and only if $D_{\gamma_1\gamma_2}$ contains no dicycle. 
\end{lem}

\begin{proof}
We first prove the sufficiency.
If $D_{\gamma_1\gamma_2}$ contains no dicycle, then it admits a topological ordering $v_1,v_2,\ldots,v_{|V(G)|}$, such that $v_k$ has no out-neighbor in $\{v_1,v_2,\ldots,v_k\}$ for each $k\in \{1,2,\ldots,|V(G)|\}$. By processing the vertices in reverse topological order, we recolor each vertex $v$ from $\gamma_1(v)$ to $\gamma_2(v)$, thereby completing the transformation.

Then we prove the necessity by contradiction. 
Suppose that $D_{\gamma_1\gamma_2}$ contains a dicycle, denoted by $C$. Let $v_i$ be the vertex in $V(C)$ that is recolored first in the transformation from $\gamma_1$ to $\gamma_2$. Since there exists a vertex $v_{j}\in V(C)$ such that $\gamma_2(v_i)=\gamma_1(v_{j})$, 
the vertex $v_i$ can not be recolored from $\gamma_1(v_i)$ to $\gamma_2(v_i)$ in one step, a contradiction.
\end{proof}
%\lh{For convenience, we restate Theorem \ref{cactus} below as follows.}

%\begin{thmbis}{cactus}
%Let $G$ be a cactus graph and $\ell$ be an integer with $\ell\geq 4$. For every pair of $\ell$-colorings $\alpha$ and $\beta$, there is a transformation from $\alpha$ to $\beta$ by recoloring each vertex at most $3$ times. 
%\end{thmbis}
%\lzx{Repeat? and Thm 1.2 below?}
 
\begin{proof}[\noindent{\bf Proof of Theorem~\ref{cactus}.}]
Note that $\mathcal{R}_{\ell}(G)$ is connected for $\ell \geq 4$ (see \cite{Dyerdeg2006,Cerecedamix2009}). Let $\mathcal{P}$ denote the set of all paths from $\alpha$ to $\beta$ in $\mathcal{R}_{\ell}(G)$. For a path $P=\alpha\cdots\beta\in \mathcal{P}$ whose vertices represent colorings of $G$, we can partition $P$ into $i_P$ segments $b_0\cdots b_1, b_1\cdots b_2, \ldots, b_{i_{P-1}}\cdots b_{i_P}$, where $b_0=\alpha, b_{i_P}=\beta$, such that every vertex $v\in V(G)$ is recolored at most once from $b_i$ to $b_{i+1}$ for each $i\in \{0,1,\ldots,i_{P-1}\}$. We call $b_0,b_1,\ldots,b_{i_P}$ a {\it step transformation} of $P$ and $i_P$ a {\it step transformation number}. In order to prove Theorem~\ref{cactus}, it suffices to show that there exists a step transformation from $\alpha$ to $\beta$ with step transformation number at most $3$.

We proceed by induction on the number $q$ of blocks of $G$.
% that there exists a step transformation of $G$ with step transformation number at most $3$.
If $q=1$, then $G$ is either an edge or a cycle. The edge case is trivial, so we assume that $G$ is a cycle. In this case, we construct a step transformation $\alpha,c_1,c_2,\beta$ such that $D_{c_1\beta}$ contains no digon and $D_{c_2\beta}$ contains no dicycle.
Let $G=x_0x_1\ldots x_nx_0$ be a cycle with $n\geq 2$. In the following, the subscripts of $x$ are taken modulo $n+1$. We begin by constructing a coloring $c_1$ such that $D_{c_1\beta}$ contains no digon. If $D_{\alpha\beta}$ contains no digon, then let $c_1=\alpha$. Otherwise, we modify $\alpha$ to obtain $c_1$ as follows.
For each maximal digon trail $x_{k}x_{k+1}\ldots x_{k+p}$, where the vertices appear in order along $G$ and each $x_{k+{i}}x_{k+{i+1}}$ forms a digon for $i\in \{0,1,\ldots,p-1\}$, recolor $x_{k+j}$ to a color in $\{1,2,\ldots,\ell\}\setminus\{\alpha(x_{k+j-1}), \alpha(x_{k+j}), \alpha(x_{k+j+1})\} $ for each $j\equiv 1 \pmod 2$ and $1\leq j\leq p$. Moreover, if $p$ is odd, then $x_{k+{p}}$ is among the recolored vertices. If it is also the tail of an out-arc $x_{k+p}x_{k+p+1}$ in $D_{\alpha\beta}$, which implies $\alpha(x_{k+p-1})=\beta(x_{k+p})=\alpha(x_{k+{p+1}})$, then we further exclude  $\beta(x_{k+{p+1}})$ from the color choices for $x_{k+{p}}$. It is worth noting that $D_{\alpha\beta}$ can form a complete dicycle only if $G$ is an even cycle, in which case the above recoloring procedure remains valid. By construction, each recolored vertex has no in-neighbor in $D_{c_1\beta}$, and hence $D_{c_1\beta}$ contains no digon.
Next, we construct a coloring $c_2$ from $c_1$ such that $D_{c_2\beta}$ contains no dicycle. If $D_{c_1\beta}$ is not a dicycle, then let $c_2=c_1$. Otherwise, we may assume $D_{c_1\beta}=x_0x_1\ldots x_nx_0$ is a dicycle, and we modify $c_1$ to obtain $c_2$ as follows. Recolor $x_0$ with a color in $\{1,2,\ldots,\ell\} \setminus \{c_1(x_n),c_1(x_0),c_1(x_1)\}$. Furthermore, if this recoloring results in a digon $x_0x_1$, which implies $x_0$ is recolored to $c_1(x_2)$, then we recolor $x_1$ with a color in $\{1,2,\ldots,\ell\}\setminus \{c_1(x_1),c_1(x_2),\beta(x_2)\}$. By the recoloring operations, there is no arc between $x_n$ and $x_0$ in $D_{c_2\beta}$, and exactly one arc between $x_i$ and $x_{i+1}$ for each $i\in\{0,1\}$. Thus, 
$D_{c_2\beta}$ has no dicycle. By construction and Lemma~\ref{Lemma-j=i+1}, each vertex of $G$ is recolored at most once in each of the transformations from $\alpha$ to $c_1$, from $c_1$ to $c_2$, and from $c_2$ to $\beta$. Therefore, $\alpha,c_1,c_2,\beta$ form a step transformation.

Hence, we may assume that $q\geq 2$. This implies that $G$ contains end blocks. Let $B$ be an end block and $u\in V(B)$ be a cut vertex of $G$. By induction hypothesis on $G'=G-(V(B)\setminus\{u\})$, there exists a step transformation denoted by $c'_0,c'_1,c'_2,c'_3$, where $c'_0=\alpha|_{G'}$ and $c'_3=\beta|_{G'}$. Note that it is possible that $c'_i=c'_{i+1}$ for some $i\in\{0,1,2\}$. To obtain the desired step transformation on the entire graph $G$, it suffices to construct a step transformation $c''_0,c''_1,c''_2,c''_3$ on $B$ such that $c''_0=\alpha|_{B}$, $c''_3=\beta|_{B}$, and $c''_i(u)=c'_i(u)$
for all $i\in \{0,1,2,3\}$. This suffices because 
if both $D_{c'_ic'_{i+1}}$ and $D_{c''_ic''_{i+1}}$ contain no dicycle for each $i\in\{0,1,2\}$, then the combined digraph also contains no dicycle.
By Lemma~\ref{Lemma-j=i+1}, this ensures a valid step transformation on $G$.

% To complete the proof, we shall construct a step transformation on the entire graph $G$, denoted by $\alpha,c_a,c_b,\beta$, such that $c_a|_{G'}=c'_1$ and $c_b|_{G'}=c'_2$. 
% By definition, $B$ is either an edge or a cycle.
% For the former case, assume $B=uv$. We
% construct a coloring $c_a$ from $c_1'$ by recoloring $v$ with a color in $\{1,2,\ldots,\ell\}\setminus \{\alpha(u),c_1'(u),c_2'(u)\}$, and a coloring $c_b$ from $c_2'$ by recoloring $v$ with a color in $\{1,2,\ldots,\ell\}\setminus \{c_2'(u),\beta(u)\}$. By Lemma \ref{Lemma-j=i+1}, $D_{c'_ic'_{i+1}}$ has no dicycle for each $i\in\{0,1,2\}$. The same holds for 
% $D_{\alpha c_{a}}$, $D_{c_{a}c_{b}}$, and $D_{c_{b}\beta}$, due to the recoloring operations on $v$. Therefore, by Lemma \ref{Lemma-j=i+1} again, we conclude that $\alpha, c_a, c_b, \beta$ form a step transformation.
By definition, $B$ is either an edge or a cycle. For the former case, assume $B=uv$.
Choose $c''_1(v)$ to be a color in $\{1,2,\ldots,\ell\}\setminus \{c_0''(u),c_1''(u),c_2''(u)\}$ and choose $c''_2(v)$
from $\{1,2,\ldots,\ell\}\setminus \{c_2''(u),c_3''(u)\}$. With this choice, it follows that $D_{c''_ic''_{i+1}}$ has no dicycle for each $i\in\{0,1,2\}$. Therefore, the desired step transformation on $B$ is established. Hence, we may assume that
$B$ is a cycle, say $uv_1v_2\ldots v_nu$. The proof is split into two cases.
% In this case, we also construct $c_a$ and $c''_2$ from $c'_1$ and $c'_2$, respectively, to obtian the desired step transformation. Similar to the edge case, it suffices to gurantee that $D_{c''_ic''_{i+1}}$ has no dicycle by Lemma \ref{Lemma-j=i+1}, where $c''_0=\alpha|_B$, $c''_1=c_a|_B$, $c''_2=c''_2|_B$, and $c''_3=\beta|_B$

% Now we get a transformation of $u$: $\alpha(u),c_a(u),c''_2(u),\beta(u)$. Every time when $u$ is recolored in $G'$, assign a new coloring to $B$ at this time. We next describe the recoloring procedure for the block $B$.

{\bf Case 1.} $|V(B)|=3$.

We first consider the case that  $|\{c''_0(v_1),c''_0(v_2),c''_0(u),c''_2(u)\}|\leq 3$ or $c''_0(v_i)\neq c''_1(u)$ for each $i\in\{1,2\}$. 
If $c''_0(v_i)\neq c''_1(u)$ for each $i\in\{1,2\}$, then we construct a coloring $c''_1$ from $c''_0$
by recoloring  $v_1$ with a color $c''_1(v_1)\in \{1,2,\ldots,\ell\}\setminus \{c''_0(v_2),c''_1(u),c''_2(u)\}$, and $v_2$ with $c''_1(v_2)\in \{1,2,\ldots,\ell\}\setminus \{c''_1(v_1),c''_1(u),c''_2(u)\}$;
%%%%%%%%%%%%%%%%%%%%%%%%%%%%%%%%%%%%%%%%%%%%%%%%%%%%%
if $c''_0(v_i)= c''_1(u)$ for some $i\in\{1,2\}$, say $v_1$, and $|\{c''_0(v_1),c''_0(v_2),c''_0(u),c''_2(u)\}|\leq 3$, then recolor $v_1$ with $c''_1(v_1)\in \{1,2,\ldots,\ell\}\setminus \{c''_0(v_1),c''_0(v_2),c''_0(u),c''_2(u)\}$, and $v_2$ with $c''_1(v_2)\in \{1,2,\ldots,\ell\}\setminus \{c''_1(v_1), c''_1(u),c''_2(u)\}$. In both cases above, we obtain a coloring $c''_1$ of $B$ such that between any pair of vertices in $D_{c''_0c''_1}$,  there is at most one arc, and $v_2$ has no in-neighbor. This implies that
$D_{c''_0c''_1}$ contains no dicycle. Note that  $c''_1(v_i)\neq c''_2(u)$ for each $i\in\{1,2\}$. Now, we 
construct a coloring $c''_2$ from $c''_1$ as follows.
First, recolor $v_1$ with a color $c''_2(v_1)\in \{1,2,\ldots,\ell\}\setminus \{c''_1(v_2),c''_2(u),c''_3(u)\}$.
If $|\{c''_2(v_1),c''_3(v_1),c''_2(u),c''_3(u)\}|\leq 3$, then recolor $v_2$ with $c''_2(v_2)\in\{1,2,\ldots,\ell\}\setminus \{c''_2(v_1),c''_3(v_1),c''_2(u),c''_3(u)\}$; if $|\{c''_2(v_1),c''_3(v_1),c''_2(u),c''_3(u)\}|=4$,  then  recolor $v_2$ with $c''_2(v_2)=c''_3(u)$ when $c''_3(v_2)=c''_2(v_1)$, and with $c''_2(v_2)=c''_3(v_1)$ when $c''_3(v_2)\neq c''_2(v_1)$. In all cases above, we obtain a coloring $c''_2$ of $B$ such that in $D_{c''_1c''_2}$,  there is at most one arc between any pair of vertices, and $u$ has no out-neighbor. Similarly, in $D_{c''_2c''_3}$,  there is at most one arc between any pair of vertices. If $|\{c''_2(v_1),c''_3(v_1),c''_2(u),c''_3(u)\}|\leq 3$, then $v_2$ has no in-neighbor; and if $|\{c''_2(v_1),c''_3(v_1),c''_2(u),c''_3(u)\}|=4$, then there is no arc between $u$ and $v_1$. This implies that both $D_{c''_1c''_2}$ and $D_{c''_2c''_3}$ contain no dicycle.
By Lemma~\ref{Lemma-j=i+1}, $c''_0,c''_1,c''_2,c''_3$ is a step transformation on $B$.

Hence, we may assume that $|\{c''_0(v_1),c''_0(v_2),c''_0(u),c''_2(u)\}|= 4$ and $c''_0(v_i)= c''_1(u)$ for some $i\in\{1,2\}$, say $v_1$. Note that $c''_2(u)\notin \{c''_0(v_1),c''_0(v_2),c''_0(u)\}$ and $c''_0(v_2)\neq c''_1(u)$.
We first construct $c''_1$ from $c''_0$ by recoloring $v_1$ with $c''_1(v_1)=c''_2(u)$, and recoloring $v_2$ with
$c''_1(v_2)=c''_0(u)$ if $c''_3(u)=c''_0(u)$, or $c''_1(v_2)=c''_0(v_2)$ otherwise. 
By construction, in $D_{c''_0c''_1}$, there is at most one arc between any pair of vertices, and $v_1$ has no out-neighbor. This implies that
$D_{c''_0c''_1}$ contains no dicycle. Note that  $c''_1(v_1)=c''_2(u)$ and $c''_1(v_2)\neq c''_2(u)$. Now, we 
construct a coloring $c''_2$ from $c''_1$ as follows.
First, recolor $v_1$ with $c''_2(v_1)\in \{c''_0(v_1),c''_0(v_2),c''_0(u),c''_2(u)\}\setminus \{c''_1(v_1),c''_1(v_2),c''_1(u)\}$. If $|\{c''_2(v_1),c''_2(u),c''_3(v_1),c''_3(u)\}|\leq 3$, then recolor $v_2$ with $c''_2(v_2)\in \{1,2,\ldots,\ell\}\setminus \{c''_2(v_1),c''_2(u),c''_3(v_1),c''_3(u)\}$; if $|\{c''_2(v_1),c''_2(u),c''_3(v_1),c''_3(u)\}|=4$, then recolor $v_2$ with $c''_2(v_2)=c''_3(u)$ if $c''_3(v_2)=c''_2(v_1)$, and with $c''_2(v_2)=c''_3(v_1)$ if $c''_3(v_2)\neq c''_2(v_1)$.  
In all cases above, we obtain a coloring $c''_2$ of $B$ such that in $D_{c''_1c''_2}$,  there is at most one arc between any pair of vertices, and $v_1$ has no out-neighbor. Moreover, in $D_{c''_2c''_3}$,  there is at most one arc between any pair of vertices. If $|\{c''_2(v_1),c''_3(v_1),c''_2(u),c''_3(u)\}|\leq 3$, then $v_2$ has no in-neighbor; and if $|\{c''_2(v_1),c''_3(v_1),c''_2(u),c''_3(u)\}|=4$, then there is no arc between $u$ and $v_1$. This implies that both $D_{c''_1c''_2}$ and $D_{c''_2c''_3}$ contain no dicycle.
By Lemma~\ref{Lemma-j=i+1}, $c''_0,c''_1,c''_2,c''_3$ is a step transformation on $B$.

%Consequently, $D_{bh}$ is a union of dipaths and isolated vertices, and we get $h\leq 3$ by Lemma~\ref{Lemma-j=i+1}.

{\bf Case 2.} $|V(B)|\geq 4$.

To facilitate the construction of a coloring $c''_1$ of $B$ from $c''_0$, we define a {\it good recoloring operation} on $v_i$, for the triples $(i,j,k) = (1,2,3)$ and $(i,j,k) = (n, n-1, n-2)$, as follows.
 
\begin{itemize}
    \item set $c''_1(v_i)\in \{1,2,\ldots,\ell\}\setminus \{c''_1(u),c''_2(u),c''_0(v_j)\}$ if $c''_0(v_i)\neq c''_1(u)$;
    \item set $c''_1(v_i)\in \{1,2,\ldots,\ell\}\setminus \{c''_0(u),c''_1(u), c''_2(u),c''_0(v_j)\}$ if  $c''_0(v_i)= c''_1(u)$ and $|\{c''_0(u),c''_1(u), \allowbreak c''_2(u),c''_0(v_j)\}|\leq3$;
    \item set $c''_1(v_i)=c''_0(v_j)$ and $c''_1(v_j)\in \{1,2,\ldots,\ell\}\setminus \{c''_0(v_i),c''_0(v_j),c''_0(v_k)\}$ if $c''_0(v_i)= c''_1(u)$ and $|\{c''_0(u),c''_1(u),c''_2(u),c''_0(v_j)\}|=4$.
\end{itemize}

Now, we construct a coloring $c''_1$ of $B$ from $c''_0$ such that $D_{c''_0c''_1}$ has no dicycle and $c''_1(v_i)\neq c''_2(u)$ for each $i\in\{1,n\}$ as follows. If $|V(B)|\geq 5$, then we take the good recoloring operations on $v_1$ and $v_n$. For each triple $(i,j,k)\in \{(1,2,3),(n, n-1, n-2)\}$, there is at most one arc between any pair in $\{u,v_i,v_j,v_k\}$ in $D_{c''_0c''_1}$. Moreover, $v_i$ has no in-neighbor if the first item of the good recoloring operation is applied, no out-neighbor if the second, and $v_j$ has no out-neighbor if the third. In particular, the construction remains valid when $|V(B)|= 5$, since at most one vertex in $\{v_2, v_{n-1}\}$ is recolored. This is ensured by the condition that the sets $\{c''_0(u),c''_1(u),c''_2(u),c''_0(v_j)\}$ differ for $j=2$ and $j=n-1$. Hence, $D_{c''_0c''_1}$ has no dicycle.
For the case $|V(B)|=4$, if $|\{c''_0(u),c''_1(u),c''_2(u),c''_0(v_2)\}|\leq 3$, then we take the good recoloring operations on $v_1$ and $v_3$; if $|\{c''_0(u),c''_1(u),c''_2(u),c''_0(v_2)\}|= 4$, then 
we set $c''_1(v_1)=c''_1(v_3)=c''_0(v_2)$ and choose $c''_1(v_2)\in \{1,2,\ldots,\ell\}\setminus \{c''_0(v_1),c''_0(v_2),c''_0(v_3)\}$. In both cases, 
there is at most one arc between any pair in $\{u,v_1,v_2,v_3\}$ in $D_{c''_0c''_1}$. Moreover, in the former case, both of $v_1$ and $v_3$ have no in- or out-neighbor, and in the latter, $v_2$ has no out-neighbor. Thus, $D_{c''_0c''_1}$ has no dicycle.

%If $c''_0(v_2)\neq c''_0(u)$, then we get that these at least three options for $c''_1(v_1)$ are pairwise distinct; if $c''_0(v_2)=c''_0(u)$, then the third case cannot appear, since there are at least two options for $c''_1(v_1)$ in the second case, considering the option for $c''_1(v_1)$ in the first case, these at least three options for $c''_1(v_1)$ are pairwise distinct. 
%We say an option is a {\it good option for $v_1$} if this option can recolor $v_1$ to $c''_1(v_1)\notin\{c''_0(u),c''_1(u),c''_2(u)\}$ by recoloring $v_1$ and $v_2$ at most once. A good option for $v_1$ always exists since we have at least three kinds of choices for color of $v_1$ which is not $c''_0(u)$ in these options. Similarly, we can define {\it good option for $v_n$}
% When $|V(B)|\geq 5$, we do a good option for $v_1$, then do a good option for $v_n$, finally recolor $u$ to $c''_1(u)$ to get $c''_1$; When $|V(B)|=4$, we do a good option for $v_1$, then recolor $v_3$ to $c''_1(v_1)$, finally recolor $u$ to $c''_1(u)$ to get $c''_1$. This $c''_1$ satisfies our need. 

% With a similar argument to Case 1, we construct a coloring $c''_1$ from $c''_0$ as follows. First, for each $i\in\{1,n\}$, recolor $v_i$ with a color  $c''_1(v_i)$ in $\{1,2,\ldots,\ell\}\setminus\{c''_0(u),c''_1(u),c''_2(u)\}$. If $|V(B)|\leq5$, then recolor all vertices of $V(B)\setminus \{u,v_1,v_n\}$ with distinct colors in $\{1,2,\ldots,\ell\}\setminus\{c''_1(v_1),c''_1(v_n)\}$; if $|V(B)|\geq6$, then

By Lemma \ref{Lemma-j=i+1}, to complete the proof, it suffices to construct a coloring $c''_2$ of $B$ from $c''_1$ such that both $D_{c''_1c''_2}$ and $D_{c''_2c''_3}$ contain no dicycle. Without loss of generality, assume that $c''_3(v_n)=c''_2(u)$ if
$c''_2(u)\in\{c''_3(v_1),c''_3(v_n)\}$.
We begin by recoloring $v_1$ with a color $c''_2(v_1)\in \{1,2,\ldots,\ell\} \setminus \{c''_2(u),c''_1(v_2),c''_3(u)\}$. Then, for each $i\in \{2,3,\ldots,n-1\}$, we recolor $v_i$ with $c''_2(v_i)\in \{1,2,\ldots,\ell\}\setminus \{c''_2(v_{i-1}),c''_1(v_{i+1}),c''_3(v_{i-1})\}$. If we choose a color $c''_2(v_n)$ for $v_n$ from 
the color set $\{1,2,\ldots,\ell\}\setminus\{c''_2(u),c''_2(v_{n-1})\}$, then under this construction, $D_{c''_1c''_2}$ contains no dicycle, whereas $D_{c''_2c''_3}$ may contain a dicycle, as a
digon between $u$ and $v_n$ or between $v_{n-1}$ and $v_n$, or as a spanning dicycle along $uv_n\cdots v_1u$. To avoid this, we shall further restrict the available colors for $v_n$ by removing more colors from the color set $\{1,2,\ldots,\ell\}\setminus\{c''_2(u),c''_2(v_{n-1})\}$.
 
If $c''_2(u)\notin\{c''_3(v_1),c''_3(v_n)\}$, then choose $c''_2(v_n)\in\{1,2,\ldots,\ell\}\setminus\{c''_2(u),c''_2(v_{n-1}),c''_3(v_{n-1})\}$. Under this construction, $D_{c''_2c''_3}$ contains no dicycle, since $u$ has no in-neighbor and the arc $v_{n-1}v_n$ does not exist. Hence, we may assume $c''_2(u)\in\{c''_3(v_1),c''_3(v_n)\}$, which implies $c''_3(v_n)=c''_2(u)$ by assumption. If the set $\{1,2,\ldots,\ell\}\setminus\{c''_2(v_{n-1}),c''_3(v_{n-1}),c''_2(u),c''_3(u)\}$ is not empty, then we recolor $v_n$ with a color $c''_2(v_n)$ from this set. Again, in this case, $D_{c''_2c''_3}$ contains no dicycle, since $v_n$ has no in-neighbor. Therefore, it remains to consider the case where $\ell=|\{c''_2(v_{n-1}),c''_3(v_{n-1}),c''_2(u),c''_3(u)\}|=4$ and $c''_3(v_n)=c''_2(u)$. In this case, we recolor $v_n$ with the color $c''_2(v_n)=c''_3(v_{n-1})\in \{1,2,\ldots,\ell\}\setminus\{c''_2(v_{n-1}),c''_2(u),c''_3(u)\}$. Under this construction, $D_{c''_2c''_3}$ contains no dicycle, since $u$ has no out-neighbor and the arc $v_nv_{n-1}$ does not exist.

This completes the proof of Theorem~\ref{cactus}.
\end{proof}

\section{Proof of Theorem~\ref{K23}}\label{last}

In this section, we shall prove Theorem~\ref{K23}, beginning with several key lemmas.

\begin{lem}\label{Lemma-ell>5}
Let $\ell\geq 6$ and $q\leq3$ be integers, and let $H$ be either an outerplanar graph or $K_4$ with $V(H)=\{v_1,v_2,\ldots,v_n\}$, where $n\geq 2$. Suppose that the vertex $v_n$ is to be recolored in the order $c_0(v_n), c_1(v_n),\ldots,c_q(v_n)$.
Then for every pair of $\ell$-colorings $\alpha'$ and $\beta'$ of $H$ satisfying $\alpha' (v_n)=c_0(v_n)$ and $\beta'(v_n)=c_q(v_n)$, there exists a transformation from $\alpha'$ to $\beta'$, in which each vertex is recolored at most $3$ times and $v_n$ follows the prescribed coloring order.
\end{lem}

\begin{proof}
We first consider the case that $H$ is $K_4$, so $n=4$. We perform the following operations prior to the recoloring of $v_4$ from $c_0(v_4)$ to $c_1(v_4)$. If $c_1(v_4)=c_0(v_i)$ for some $i\in \{1,2,3\}$, then recolor $v_i$ with a color in $\{1,2,\ldots,\ell\} \setminus \{\alpha'(v_1),\alpha'(v_2),\alpha'(v_3),\alpha'(v_4),c_2(v_4)\}$ (note that $c_2(v_4)$ may not exist).  Repeat above operations prior to each recoloring of $v_4$ from $c_{j}(v_4)$ to $c_{j+1}(v_4)$ for each $j\in \{1,\ldots,q-1\}$.
Since $q\leq 3$, each vertex in $\{v_1,v_2,v_3\}$ is recolored at most twice, and the total number of recolorings among them is at most three. With loss of generality, we may assume that $v_1$ has been recolored at most twice, and each vertex in $\{v_2,v_3\}$ at most once. Let $c'(v_1),c'(v_2),c'(v_3),c'(v_4)$ be the current colors of $v_1,v_2,v_3,v_4$, respectively, where $c'(v_4)=c_q(v_4)=\beta'(v_4)$.
If $\beta'(v_1)=c'(v_j)$ for some $j\in \{2,3\}$, then we recolor $v_j$ with a color in $\{1,2,\ldots,\ell\} \setminus \{c'(v_1),c'(v_2),c'(v_3),\beta'(v_{5-j}),\beta'(v_4)\}$, $v_1$ with $\beta(v_1)$, $v_{5-j}$ with $\beta(v_{5-j})$, and finally $v_j$ with $\beta(v_j)$, in order.
If $\beta(v_1)\neq c'(v_j)$ for each $j\in \{2,3\}$, then we recolor $v_1$ with $\beta(v_1)$, $v_2$ with a color in $\{1,2,\ldots,\ell\} \setminus \{\beta'(v_1),\beta'(v_2),\beta'(v_3),\beta'(v_4),c'(v_3)\}$, $v_3$ with $\beta'(v_3)$, and finally $v_2$ with $\beta'(v_2)$, in order. In both cases, we obtain a transformation from $\alpha'$ to $\beta'$, in which each vertex is recolored at most $3$ times and $v_4$ follows the prescribed coloring order.

Now, we consider the case that $H$ is an outerplanar graph. We proceed by induction on $n$. The case $n\leq 2$ is trivial. Suppose the statement  holds for all outerplanar graphs with at most 
$n-1$ vertices. Now we consider $|V(H)|= n$, where $n\geq 3$. Since each outerplanar graph has at least two vertices with degree at most $2$,  we can choose a vertex $v_1\neq v_n$ in $H$ such that $d(v_1)\leq 2$. By induction hypothesis on $H'=H-v_1$, each vertex of $H'$ is recolored at most $3$ times in the transformation from $\alpha'|_{H'}$ to $\beta'|_{H'}$. It remains to consider the recoloring of $v_1$. Since $d(v_1)\leq 2$, the number of unused colors on $N(v_1)$ is at least $\ell-2$. We recolor $v_1$ with one of these colors that is not the target color of any vertex in $N(v_1)$
in the next $\ell-3$ recolorings. Since each vertex in $N(v_1)$ is recolored at most $3$ times, the total number of color conflicts affecting $v_1$ is at most $3|N(v_1)|$. Therefore, the number of times we recolor $v_1$ is at most $\lceil{\frac{3|N(v_1)|}{\ell-3}}\rceil+1$, which is at most $3$. This completes the argument.
\end{proof}

\begin{proof}[\noindent{\bf Proof of Theorem~\ref{K23}(\ref{k23-6}).}]
We proceed by induction on the number of blocks of $G$, denoted by $p$.
If $p=1$, then $G$ is either a copy of $K_4$ or an outerplanar graph. In this case, the result follows directly from Lemma \ref{Lemma-ell>5}.
Hence, we may assume $p\geq 2$. Let $B$ be an end block of $G$ and $u\in V(B)$ be a cut vertex of $G$, which implies that $B$ is a copy of $K_4$ or an outerplanar graph. By induction hypothesis on $G'=G-(V(B)\setminus\{u\})$, we know that there exists a transformation of $G'$ from $\alpha|_{G'}$ to $\beta|_{G'}$ in which  each vertex is recolored at most $3$ times. 
This transformation preserves the prescribed coloring order of $u$,  so Lemma \ref{Lemma-ell>5} can be applied to extend the transformation on $B$. This completes the proof of Theorem~\ref{K23}(\ref{k23-6}).
\end{proof}

To complete the proof of Theorem~\ref{K23}(\ref{k23-5}), we shall show that it suffices to prove the following lemma, whose proof will be given at the end of this section. 

\begin{lem}\label{lem: chordal 5 to 5}
Let $H$ be a $K_{2,3}$-minor-free chordal graph. For every pair of $5$-colorings $\alpha'$ and $\beta'$, there exists a transformation from $\alpha'$ to $\beta'$ by recoloring each vertex at most $h$ times, where $h=692$. Moreover, in this transformation, for each vertex $v\in V(H)$, if $v$ is a vertex in a copy of $K_4$ but not in a block that is an outerplanar subgraph in $G$, then $v$ is recolored at most $h-3$ times.
\end{lem}

To facilitate the application of Lemma \ref{lem: chordal 5 to 5} to Theorem~\ref{K23}(\ref{k23-5}) more conveniently and directly, we derive the following corollary, where a $4$-coloring is a coloring restricted in the color set $\{1,2,3,4\}$.

\begin{corollary}\label{cor: 5 to 4}
Let $H^*$ be a $K_{2,3}$-minor-free graph and $\alpha^*$ be a $5$-coloring of $H^*$. Then there exists a $4$-coloring $\beta^*$ of $H^*$, such that there is a transformation from $\alpha^*$ to $\beta^*$ by recoloring each vertex at most $h$ times, where $h=692$.
\end{corollary}

\begin{proof}
We perform the following operations on $H^*$. For each block of $H^*$ that is an outerplanar graph, we add as many edges as possible such that each face, except for the exterior face, is bounded by a triangle, and it is still outerplanar.  Denote the resulting graph by $H_1$, which is a $K_{2,3}$-minor-free chordal graph.  Note that  $\alpha^*$ induces a (possibly non-proper) $5$-coloring of $H_1$. Next, we identify  $x$ and $y$ to a single vertex if $xy\in E(H^*)$ and $\alpha^*(x)=\alpha^*(y)$.
Let $H_2$ be the resulting graph, and let $\alpha_2$ be the corresponding coloring of $H_2$ obtained from $\alpha^*$. By construction, $H_2$ is also a $K_{2,3}$-minor-free chordal graph. By Lemma~\ref{lem: chordal 5 to 5}, for any $4$-coloring $\beta_2$ of $H_2$, there exists a  transformation from $\alpha_2$ to $\beta_2$ by recoloring each vertex in $H_2$ at most $h$ times, where $h=692$. Since $\beta_2$ corresponds to a $4$-coloring $\beta^*$ of $H^*$, the transformation from $\alpha_2$ to $\beta_2$ in $H_2$ corresponds to a transformation from $\alpha^*$ to $\beta^*$ in $H^*$.\end{proof}

\begin{proof}[\noindent{\bf Proof of Theorem~\ref{K23}(\ref{k23-5}) by applying Corollary \ref{cor: 5 to 4}.}]
Let $\alpha$ and $\beta$ be two $5$-colorings of $G$. By Corollary \ref{cor: 5 to 4}, $\alpha$ can be transformed to a $4$-coloring $\alpha_1$ by recoloring each vertex at most $h$ times, where $h=692$. Similarly, we can transform $\beta$ to a $4$-coloring $\beta_1$ by recoloring each vertex at most $h$ times. By Lemma~\ref{cactus-I}, $G$ can be decomposed into an independent set $I$ and a union of cactus graphs $G-I$. Under the colorings $\alpha_1$ and $\beta_1$, we recolor every vertex in $I$ with color $5$. By Theorem~\ref{cactus},  we get a transformation on $G-I$ from $\alpha_1|_{G-I}$ to $\beta_1|_{G-I}$ by recoloring each vertex at most $3$ times, after which we recolor each vertex $v$ in $I$ with $\beta_1(v)$. Therefore, there exists a transformation from $\alpha$ to $\beta$ by recoloring each vertex at most $2h+3=1387$ times.
\end{proof}

Finally, we provide a detailed proof of Lemma~\ref{lem: chordal 5 to 5}. Before that, with a similar argument to the proof of Theorem \ref{K23}(\ref{k23-6}), we need the following lemma.

\begin{lem}\label{Lemma-chordal-5to5}
Let $h = 692$, and let $H$ be either a chordal outerplanar graph or $K_4$ with $V(H) = \{v_1, v_2, \ldots, v_n\}$, where $n \geq 2$. Suppose that the vertex $v_n$ is to be recolored in the order $c_0(v_n), c_1(v_n), \ldots, c_q(v_n)$. We assume the following:
\begin{itemize}
    \item[(1)] If $H\cong K_4$, then $q \leq h$;
    \item[(2)] If $H$ is a chordal outerplanar graph, then $q \leq h - 2$.   
\end{itemize}

Then for every pair of $5$-colorings $\alpha'$ and $\beta'$ of $H$ satisfying $\alpha'(v_n) = c_0(v_n)$ and $\beta'(v_n) = c_q(v_n)$, there exists a transformation from $\alpha'$ to $\beta'$ in which each vertex is recolored at most $h$ times, and $v_n$ follows the prescribed coloring order. Moreover, if $H \cong K_4$, then each of the vertices $v_1$, $v_2$, and $v_3$ is recolored at most $h - 3$ times in this transformation.
\end{lem}

\begin{proof}
We first consider the case that $H$ is $K_4$, so $n=4$. We perform the following operations prior to the recoloring of $v_4$ from $c_0(v_4)$ to $c_1(v_4)$. If $c_1(v_4)=c_0(v_i)$ for some $i\in \{1,2,3\}$, then recolor $v_i$ with a color in $\{1,2,\ldots,5\} \setminus \{\alpha'(v_1),\alpha'(v_2),\alpha'(v_3),c_0(v_4)\}$. Repeat above operations prior to each recoloring of $v_4$ from $c_{j}(v_4)$ to $c_{j+1}(v_4)$ for each $j\in \{1,\ldots,q-1\}$. The total number of recolorings of $v_1,v_2$, and $v_3$ is at most $q$. Since $q\leq h=692$, at most one vertex of $\{v_1,v_2,v_3\}$ can be recolored  more than $h-5$ times. If such a vertex exists, say $v_1$, then the total number of recolorings of $v_2$ and $v_3$ is at most $4$. We now revert the operations to the step just before the $\frac{h}{2}$-th recoloring of $v_1$, which would have been triggered by the upcoming recoloring of $v_4$ from $c_k(v_4)$ to $c_{k+1}(v_4)$ for some $k<q$. At this point, let the current colors of $v_1,v_2,v_3,v_4$ be $c'(v_1),c'(v_2),c'(v_3),c'(v_4)$, respectively, where $c'(v_1)=c_{k+1}(v_4)$ and $c'(v_4)=c_k(v_4)$. We then recolor $v_2$ with a color in $\{1,2,\ldots,5\} \setminus \{c'(v_1),c'(v_2),c'(v_3),c'(v_4)\}$,  $v_1$ with $c'(v_2)$, and $v_4$ with $c_{k+1}(v_4)$, in order. Note that in this modified sequence, the recolorings of $v_1$ and $v_2$ are swapped compared to the previous operations. Thus, each of $v_1$ and $v_2$ is recolored at most $\frac{h}{2}+6\leq h-5$ times.  Hence, we are reduced to the case that each vertex in $\{v_1,v_2,v_3\}$ is recolored at most $h-5$ times. 
Next, we only need to recolor $v_1,v_2,v_3$ to match $\beta|_{H-\{v_4\}}$, which can be done by recoloring each vertex at most twice. Let the current colors of $v_1,v_2,v_3,v_4$ be $c''(v_1),c''(v_2),c''(v_3),c''(v_4)$, respectively, where $c''(v_4)=\beta'(v_4)=c_q(v_4)$. Since $c''(v_1)\neq c''(v_2)\neq c''(v_3)$ and $\beta'(v_1)\neq \beta'(v_2)\neq \beta'(v_3)$, there exists a vertex $v_r$ such that $|\{c''(v_s), c''(v_t),\beta'(v_s), \beta'(v_t), \beta'(v_4)\}|\leq4$, where
$\{r,s,t\}= \{1,2,3\}$. We then recolor $v_r$ with a color $c^*(v_r)\in\{1,2,\ldots,5\}\setminus \{c''(v_s), c''(v_t),\beta'(v_s), \beta'(v_t), \beta'(v_4)\}$, $v_s$ with a color in $\{1,2,\ldots,5\}\setminus \{c^*(v_r),c''(v_t),\beta'(v_t), \beta'(v_4)\}$, $v_t$ with $\beta'(v_t)$, $v_s$ with $\beta'(v_s)$, and finally $v_r$ with $\beta'(v_r)$, in order.
Thus, we obtain a transformation from $\alpha'$ to $\beta'$, in which each vertex is recolored at most $h-3$ times and $v_4$ follows the prescribed coloring order.

The proof for the case that  
$H$ is a chordal outerplanar graph is provided in the Appendix, as it is a slight modification of the proof of Lemma 4 in \cite{Bartiertw22021}.
\end{proof}

\begin{proof}[\noindent{\bf Proof of Lemma~\ref{lem: chordal 5 to 5}.}]
 
We call a subgraph of $H$ a {\it part}, if it is a copy of $K_4$ or it is a maximal chordal outerplanar subgraph, which means that if two chordal outerplanar blocks share a cut vertex of $H$, then they belong to the same part.
We proceed by induction on the number $p$ of parts in $H$.
If $p=1$, then $H$ is a copy of $K_4$ or a chordal outerplanar graph. By Lemma \ref{Lemma-chordal-5to5}, we get the desired transformation.
So we may assume $p\geq 2$. Since $H$ is a $K_{2,3}$-minor-free chordal graph, there exists a part $H_1$ of $H$ such that $H_1$ and the remaining subgraph $H_2=H-(V(H_1)\setminus\{u\})$ share exactly one common vertex $u$, which is a cut vertex of $H$. By induction hypothesis on $H_2$, there is a transformation from $\alpha'|_{H_2}$ to $\beta'|_{H_2}$ by recoloring each vertex at most $h$ times,
where $h = 692$. In particular, if $H_1$ is a chordal outerplanar subgraph, then $u$ must be in a copy of $K_4$ but not in a block that is an outerplanar subgraph in $H_2$, and thus $u$ is recolored at most $h-3$ times in $H_2$.
Therefore, whether $H_1$ is a chordal outerplanar subgraph or a copy of $K_4$, we can apply Lemma \ref{Lemma-chordal-5to5} to obtain the desired transformation on $H$ from $\alpha'$ to $\beta'$.
\end{proof}

Summarizing the entire transformation, we shall illustrate by showing the following algorithm.

\begin{algorithm}[htbp]
    \caption{Recolor a $K_{2,3}$-minor-free graph between any two $5$-colorings}
    \begin{algorithmic}%[1] % [1] 表示行号从1开始
\STATE {\bf Input}: A $K_{2,3}$-minor-free graph $G$, two $5$-colorings $\alpha$ and $\beta$ of $G$.

\STATE {\bf Output}: A sequence of recolorings.
\FOR{each step in stage $1$}
\STATE Transform $\alpha$ to $\alpha_1$, transform $\beta$ to $\beta_1$.
\ENDFOR
\STATE Split $G$ into an independent set $I$ and a union of cactus graphs $G-I$.
\FOR{each step in stage $2$}
\STATE Recolor $I$ to the color $5$.
\ENDFOR
\FOR{each step in stage $3$}
\STATE Transform $\alpha_1|_{G-I}$ to $\beta_1|_{G-I}$.
\ENDFOR 
    \end{algorithmic}
\end{algorithm}

\subsection*{Acknowledgements}
%\noindent {\bf Acknowledgements}

\noindent The authors thank 
 Yongtang Shi for many helpful comments and suggestions, which have greatly improved the results in this paper. 
 Ruijuan Gu was partially supported by the National Natural Science Foundation of China (No. 12301460). 
 Hui Lei was partially supported by the National
Natural Science Foundation of China (Nos. 12371351, 12431013).

\newpage
\appendix
\renewcommand{\thesection}{\Alph{section}} % 设置 section 为 A、B、C...
\newtheorem{lemA}{Lemma}[section]          % 新的定理环境，依赖 section
\renewcommand{\thelemA}{\thesection.\arabic{lemA}} % 设置编号格式为 A.1, A.2

\section{Appendix}
%\section{Proof of Claim~\ref{clm: chordal outerplanar fix point h->h}}

To complete the proof of Lemma~\ref{Lemma-chordal-5to5}(2), we require an additional tool introduced in \cite{Bartiertw22021}. While their results  focus on the graphs of treewidth $2$,  we will extend these results to $K_{2,3}$-minor-free graphs with a similar argument to that used in \cite{Bartiertw22021}. 

Before presenting results, we introduce some  definitions, most of which are adapted from \cite{Bartiertw22021}. 
Let $G$ be a graph and assume that $\alpha$ and $\beta$ are two $\ell$-colorings of $G$. If $G$ can be recolored from $\alpha$ to $\beta$ by recoloring each vertex at most $h$ times, then we define a {\it recoloring sequence} $S$ as a sequence $s_1s_2\cdots s_t$, where each $s_i$ is a pair $(u_i,q_i)$.
Here, $u_i$ denotes the vertex being recolored, and $q_i$ denotes the new color assigned to $u_i$ in the $i$-th step. Note that all intermediate colorings defined by $S$ are proper colorings. For any subset $X\subseteq V(G)$, the sequence $S|_X$ represents the subsequence of $S$ restricted to recoloring vertices in a subset $X$. The sequence $S|_{v}$ represents the sequence of colors assigned to a vertex $v$ in $S$. Denoted by $|S|_{v}|$ the total number of times $v$ is recolored.

Next, we introduce how to extend a recoloring sequence by recoloring only one vertex. Assume that $S'$ is a recoloring sequence from $\alpha|_{G-v}$ to $\beta|_{G-v}$. 
Let $t_1, t_2, \ldots, t_p$ be all the steps in $S'$ where a neighbor of $v$ is recolored, and let $q_{t_i}$ denote the new color assigned at step $t_i$. A {\it valid color set} $C_{t_i}(v)\subseteq \{1,2,\ldots,\ell\}$ for $v$ at step $t_i$ is the set of colors that not used on the neighbors of $v$ after step $t_{i-1}$ and before step $t_i$. If $C_{t_i}(v)\neq \emptyset$, then a {\it best choice} of color for $v$ at step $t_i$ is determined as follows, in order of priority:

\begin{itemize}
    \item[(1)] If $\beta(v)\in C_{t_i}(v)$ and $\beta(v)\notin \{q_{t_j}|i\leq j\leq p\}$, then choose $\beta(v)$;
    \item[(2)] Otherwise, if there exists a color $q\in C_{t_i}(v)$ such that $q\notin \{q_{t_j}|i\leq j\leq p\}$, then choose such a color $q$;
    \item[(3)] Otherwise, choose a color $q\in C_{t_i}(v)$ for which the earliest index $j\geq i$ satisfying $q_{t_j}=q$ is maximized.
\end{itemize}
Now, we construct a recoloring sequence on $G$ from the recoloring sequence $S'$ on $G-v$ by applying the best choice strategy for $v$. Specifically, whenever a neighbor $v'$ of $v$ is recolored in $S'$
to the current color of $v$, we insert a recoloring step for $v$ immediately before this step of $v'$ in $S'$, assigning $v$ a best choice. Other than  these inserted steps, $v$ is not recolored at any other point, except possibly at the final step to ensure it is assigned a valid color if needed.

Let $G^*$ be a chordal graph with clique number at most $3$, and $v_1,v_2,\ldots,v_n$ be a perfect elimination ordering of $G^*$, where $v_i$ has at most $2$ neighbors in $\{v_{i+1},\ldots,v_n\}$ and these at most $3$ vertices form a clique. Based on this ordering, we construct the digraph $D(G^*)$ from $G^*$ by directing an arc from $v_i$ to $v_j$ whenever $v_i$ and $v_j$ are adjacent and $j>i$. 
Assume that $S^*=s_1s_2\cdots s_t$ is a recoloring sequence on $G^*$. For a vertex $v\in V(G^*)$, the step $i$ of the sequence $S^*$ is called {\it saved} for $v$ if $s_i$ recolors a vertex $w\in N_{D(G^*)}^+(v)$ and one of the following
holds:

$\bullet$ $v$ is not recolored at steps $1,2,\ldots,i$;

$\bullet$ $v$ is not recolored at steps $i,i+1,\ldots,t$;

$\bullet$ the two steps preceding $s_i$ in $S^*|_{N_{D(G^*)}^+[v]}$ do not recolor $v$.

\medskip

By repeatedly applying the best choice strategy to the vertices $v_{n},\ldots,v_1$ in order, Bartier et al. \cite{Bartiertw22021} constructed a special recoloring sequence in which $v_n$ is recolored at most once, from $(v_n,\alpha(v_n))$ to $(v_n,\beta(v_n))$. Based on this sequence, they established several useful lemmas (Lemmas 6 and 9 in \cite{Bartiertw22021}) that are instrumental in proving Theorem~\ref{tw2}. 
We note that the estimate in Lemma 6 of \cite{Bartiertw22021} should be understood in the ceiling form,
$|S|_v| \le 1+\left\lceil \frac{m-r}{2}\right\rceil$, as indicated by its proof. Recomputing the subsequent estimates with this form slightly changes the constants appearing later in their argument: the constants 34, 74, and 542 become 44, 94, and 692, respectively.

In our application, we need a slightly more general version in which the recoloring sequence of $v_n$ is prescribed in advance as $c_0(v_n)=\alpha(v_n), c_1(v_n),\ldots,c_q(v_n)=\beta(v_n)$, and the best choice strategy is then applied successively to the vertices $v_{n-1},\ldots,v_1$. The same arguments as those in \cite{Bartiertw22021} apply in this setting, since the proofs only use the recoloring sequence of $v_n$ as a fixed input sequence. For completeness and to make the subsequent application clear, we restate below the forms of the lemmas that we need.

% We observe that, using a similar argument, these lemmas still hold if we first fix the recoloring sequence of $v_n$ as $c_0(v_n)=\alpha(v_n),c_1(v_n),\ldots,c_q(v_n)=\beta(v_n)$, and then apply the best choice strategy to the remaining vertices $v_{n-1},\ldots,v_1$ in order.

\begin{lemA}[Lemma 6 of \cite{Bartiertw22021}]\label{lem6}
Let $G^*$ be a chordal graph with clique number at most $3$, and $v_1,v_2,\ldots,v_n$ be a perfect elimination ordering of $G^*$. Let $S^*$ be a recoloring sequence between any two $5$-colorings $\alpha$ and $\beta$ of $G^*$, constructed by first fixing the recoloring sequence of $v_n$ as $c_0(v_n)=\alpha(v_n), c_1(v_n),\ldots,c_q(v_n)=\beta(v_n)$, and then successively applying
the best choice strategy to the remaining vertices $v_{n-1},\ldots,v_1$. Then $|S^*|_v|\leq 1+\lceil {\frac{m-r}{2}}\rceil$ for any vertex $v\in V(G^*)\setminus\{v_n\}$, where $m=\sum_{w\in N^{+}(v)}|S^*|_w|$ and $r$ is the number of the save steps for $v$. 
\end{lemA}

\begin{lemA}[Lemma 9 of \cite{Bartiertw22021}]\label{lem9}
Let $G^*$ be a chordal graph with clique number at most $3$, and $v_1,v_2,\ldots,v_n$ be a perfect elimination ordering of $G^*$. Let $S^*$ be a recoloring sequence between any two $5$-colorings $\alpha$ and $\beta$ of $G^*$, constructed by first fixing the recoloring sequence of $v_n$ as $c_0(v_n)=\alpha(v_n), c_1(v_n),\ldots,c_q(v_n)=\beta(v_n)$, and then successively applying
the best choice strategy to the remaining vertices $v_{n-1},\ldots,v_1$.  Let $x$, $u$, $v$, $w$ be vertices of $G^*$ such that $N_{D(G^*)}^+(x)= \{u, v\}$ and $N_{D(G^*)}^+(u)= \{v, w\}$. Assume  $x = v_i$,  $t=max_{i'>i}|S^*|_{v_{i'}}|$. If $x$ is recolored at least $t-1$ times in $S^*$, then: 

\begin{enumerate}
    \item[(1)]{the pattern $uwvu$ appears at least $t-44$ times in $S^*|_{N^+[u]}$.}
    \item[(2)]{in the sequence of colors of $x$, there are at most $94$ indices where three consecutive colors are not pairwise distinct.}
\end{enumerate}
\end{lemA}

Combining Lemmas \ref{lem6} and \ref{lem9}, and using a similar argument as in the proof of Lemma 4 of \cite{Bartiertw22021}, we establish the following lemma.

\begin{lemA}[Lemma 4 of \cite{Bartiertw22021}]\label{LemmaA-3}
Let $h=692$, and let $G^*$ be a chordal outerplanar graph, and $v_1,v_2,\ldots,v_n$ be a perfect elimination ordering of $G^*$. Let $S^*$ be a recoloring sequence between any two $5$-colorings $\alpha$ and $\beta$ of $G^*$, constructed by first fixing the recoloring sequence of $v_n$ as $c_0(v_n)=\alpha(v_n), c_1(v_n),\ldots,c_q(v_n)=\beta(v_n)$, and then successively applying
the best choice strategy to the remaining vertices $v_{n-1},\ldots,v_1$. If $q\leq h-2$, then the recoloring sequence $S^*$ recolors each vertex at most $h$ times.
\end{lemA}

\noindent{\bf Remark.} 
In the current proof of Lemma \ref{LemmaA-3}, the upper bound $q\leq h-2$ is necessary, as the argument relies on applying Lemma \ref{lem6} to perform induction, which fails when $|S^*|_{v_n}|\geq h-1$. In this case, there may exist an in-neighbor $v$ of $v_n$ such that $|S^*|_v|=1+\lceil {\frac{m-r}{2}}\rceil\geq 1+\lceil {\frac{h+(h-1)-0}{2}}\rceil=h+1$.

%Now, we prove Lemma \ref{Lemma-chordal-5to5}(2) by applying Lemma \ref{LemmaA-3}.

\begin{proof}[\noindent{\bf Proof of Lemma~\ref{Lemma-chordal-5to5}(2) by applying Lemma \ref{LemmaA-3}.}] 
It suffices to show that for each vertex $u\in V(H)$, there exists a perfect elimination ordering of $H$ ending with $u$. Since $H$ is a chordal outerplanar graph, each block of $H$ contains at least two vertices of degree at most $2$, and for each such vertex, its closed neighborhood in the block form a clique. In particular, if $H$ has cut-vertices, then it has at least two end-blocks, each containing such a vertex  that is not a cut-vertex of $H$. We choose one of them, distinct from $u$, as $v_1$. Then $H-v_1$ remains a chordal outerplanar graph, and we repeat the argument to construct a perfect elimination ordering ending with $u$. This completes the proof.
\end{proof}

\end{document}